 \documentclass[final,1p,times,authoryear]{elsarticle}

\usepackage{amssymb}
 \usepackage{amsthm}
\usepackage{amsmath}
\usepackage{environ}
\usepackage{graphicx}
\usepackage{caption}
\usepackage{subcaption}
\usepackage{bm}

\newtheorem{theorem}{Theorem}[section]
\newtheorem{lemma}[theorem]{Lemma}
\newtheorem{cor}[theorem]{Corollary}

\theoremstyle{definition}
\newtheorem{definition}[theorem]{Definition}

\theoremstyle{remark}
\newtheorem{remark}[theorem]{Remark}

\newcommand{\Z}{\mathbb{Z}}
\newcommand{\mli}[1]{\mathit{#1}}
\newcommand{\mc}{\mathcal}
\newcommand{\mb}{\mathbb}
\DeclareMathOperator{\Tr}{Tr}

\begin{document}

\begin{frontmatter}



\title{Cauchy Identities for the Characters of the Compact Classical Groups}


\author{Amir Sepehri\tnoteref{thanks}}
\tnotetext[thanks]{The author is supported by a Weiland Graduate Fellowship in the School of Humanities and Sciences, Stanford University.}

\ead{asepehri@stanford.edu}

\address{Department of Statistics, Sequoia Hall, Stanford, CA 94305, USA}

\begin{abstract}
Motivated by statistical applications, this paper introduces Cauchy identities for characters 
of the compact classical groups. These identities generalize the well-known Cauchy identity 
for characters of the unitary group, which are Schur functions of symmetric function theory. 
Application to statistical hypothesis testing is briefly sketched.
\end{abstract}


\end{frontmatter}


\section{Introduction} 
Consider generating uniformly random three dimensional rotations, i.e.\ elements of $\mli{SO}(3)$. The standard way of generating a random three dimensional rotation is to rotate uniformly along the north pole, and then move the north pole to a uniform point on the sphere. This is a special case of the famous subgroup algorithm of \citet{diaconis1987subgroup}. However, a very common and easy to make mistake is to perform this algorithm in the reverse order; that is, first move the north pole to a random point on the sphere, and then rotate uniformly along that point. This can be rephrased in terms of axis-angle representation of three dimensional rotations. A well known fact, due to Euler, asserts that any three dimensional rotation is a pure rotation along a single fixed axis. It is common to think that the uniform distribution of the rotation induces uniform distributions on the axis and the angle, independently. However, this leads to the naive algorithm above. For instance, a paper with sampling as its main concern mentions this as an `intuitively correct' way of sampling uniformly from $\mli{SO}(3)$; even though they use another correct algorithm in a different parameterization of rotations \citep{kuffner2004effective}. Figure \ref{fig:sphere} below shows samples generated using the two algorithms above. As it can be seen, it is not easy to tell the difference between the two only by looking at this plot. However, a more careful inspection of Figure \ref{fig:sphere} may suggest that the naive sampling has more rotations with a small angle, that is, rotations with the red and black arrows `close' to each other. For a more formal inspection, consider only the rotation angles, as represented in Figure \ref{fig:hist}. Figure \ref{fig:hist} illustrates the scatter-plot of the angles for the two samplers as well as a histogram for each. It is clear that the underlying distributions are different. The scatter-plot shows that the angle for uniform sampler has a higher density around $\theta= \pi$ and a lower density around $\theta =0$. This makes intuitive sense because all the rotations with small angles are similar to the identity regardless of the axis. Therefore, a uniform angle distribution results in an oversampling of rotations closer to the identity. This is the case for the naive sampler. For a more rigorous assessment of the differences between the two samplers, formal statistical hypothesis tests can be used. If only the angle of rotation is of concern, the problem becomes a one dimensional goodness of fit testing problem. This is a classical problem with an enormous literature around it, including several volumes of textbooks. In many problems, it may not be true that the angle distribution is different between two samplers; even if it is, it might be hard to guess that by looking at the data.
\begin{figure}[t]
  \centering
  \begin{minipage}[b]{0.4\textwidth}
    \includegraphics[width=\textwidth]{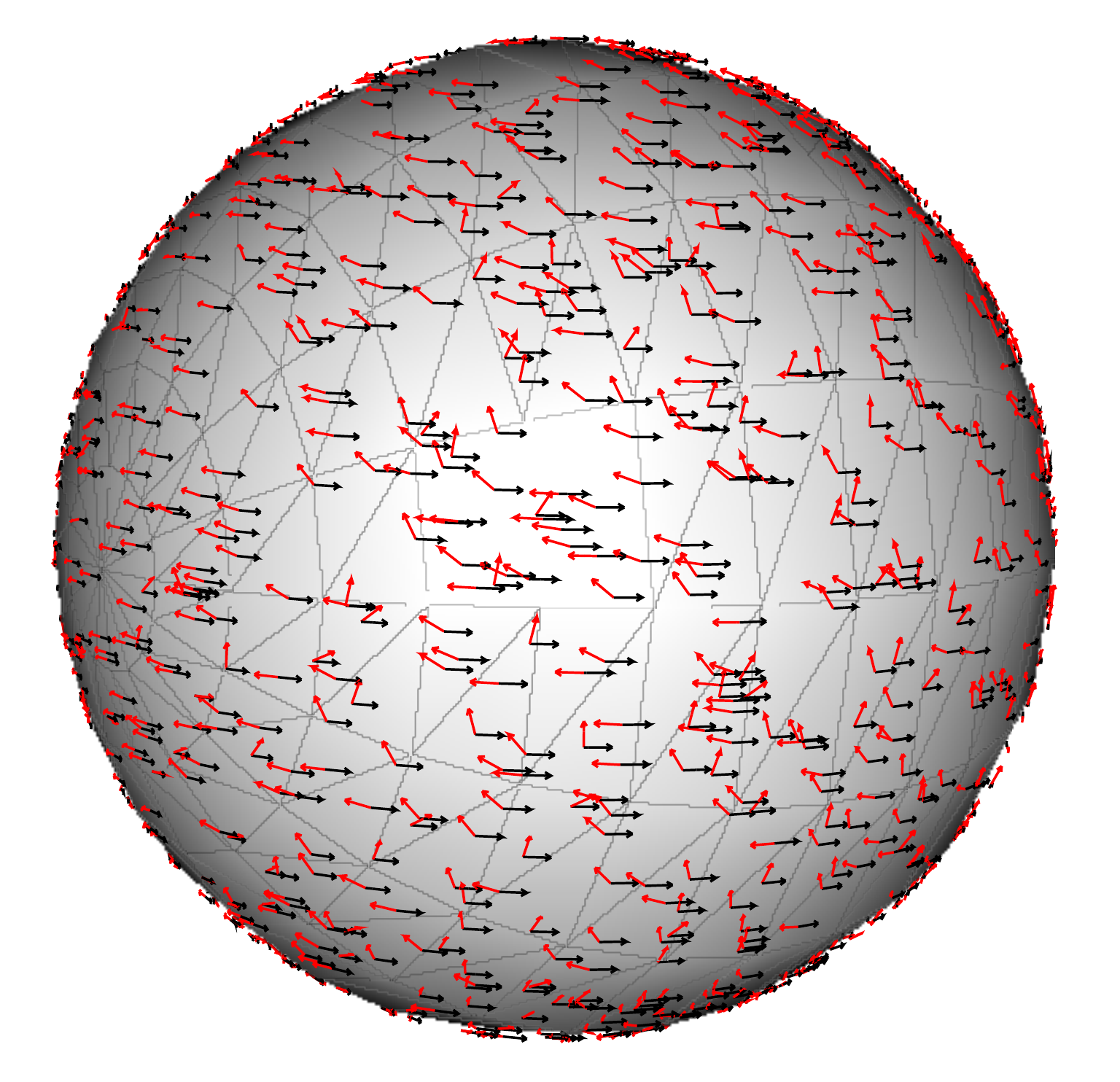}
    \caption*{{\scriptsize (a) Uniform sampling}}
  \end{minipage}
  \hfill
  \begin{minipage}[b]{0.4\textwidth}
    \includegraphics[width=\textwidth]{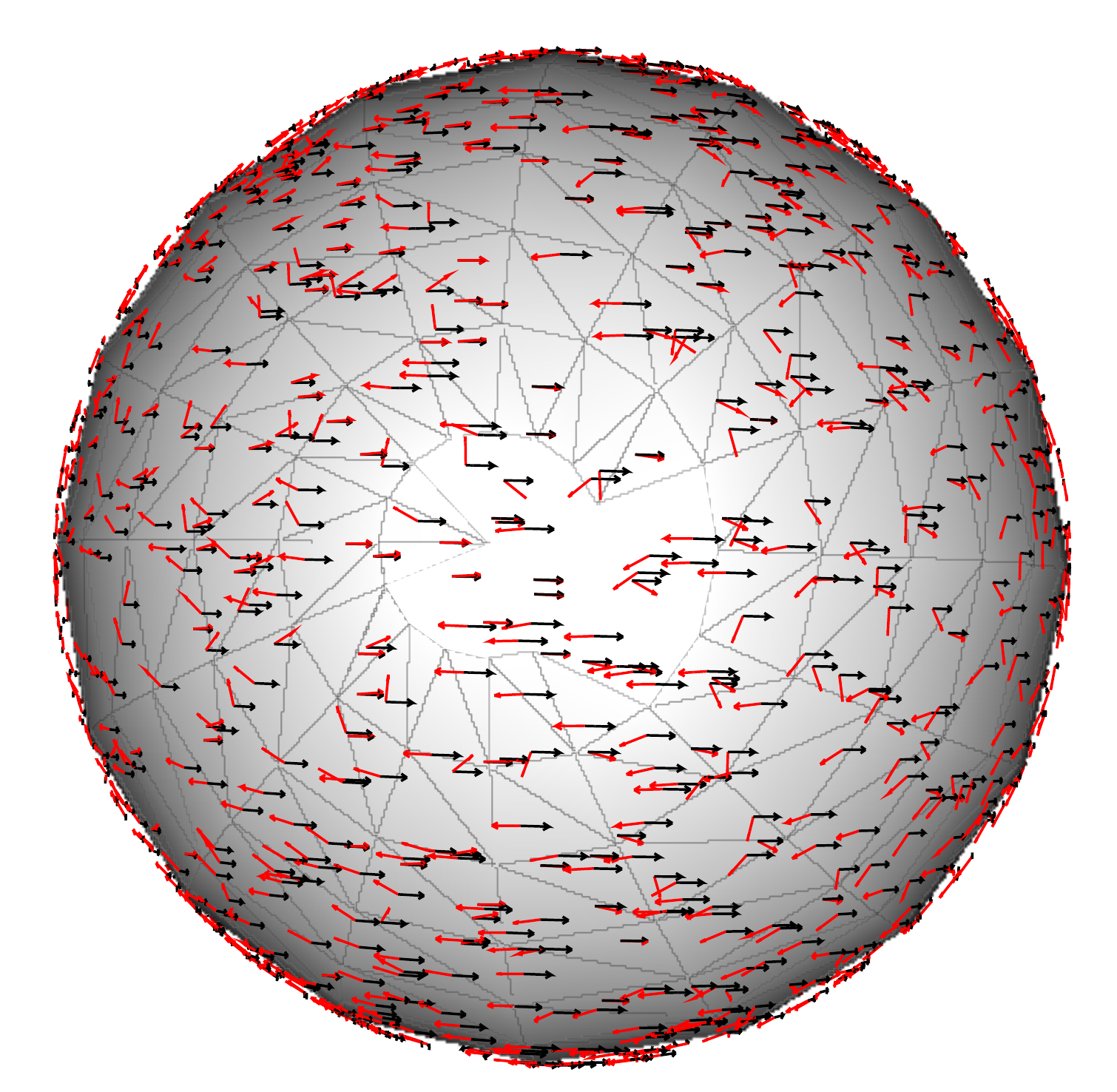}
    \caption*{{\scriptsize (b) Naive sampling}}
  \end{minipage}
  \caption{Samples from the uniform distribution (a) and naive sampling (b). Each rotation is represented by its axis and angle. The point on the sphere represents the axis, the angle between the red and black arrows show the angle (black arrow corresponds to the rotation with angle zero, i.e.\ the identity). Each plot is based on $N = 1500$ samples.}
  \label{fig:sphere}
\end{figure}
\begin{figure}[h!]
  \centering
  \begin{minipage}[b]{0.35\textwidth}
   \includegraphics[width=\textwidth]{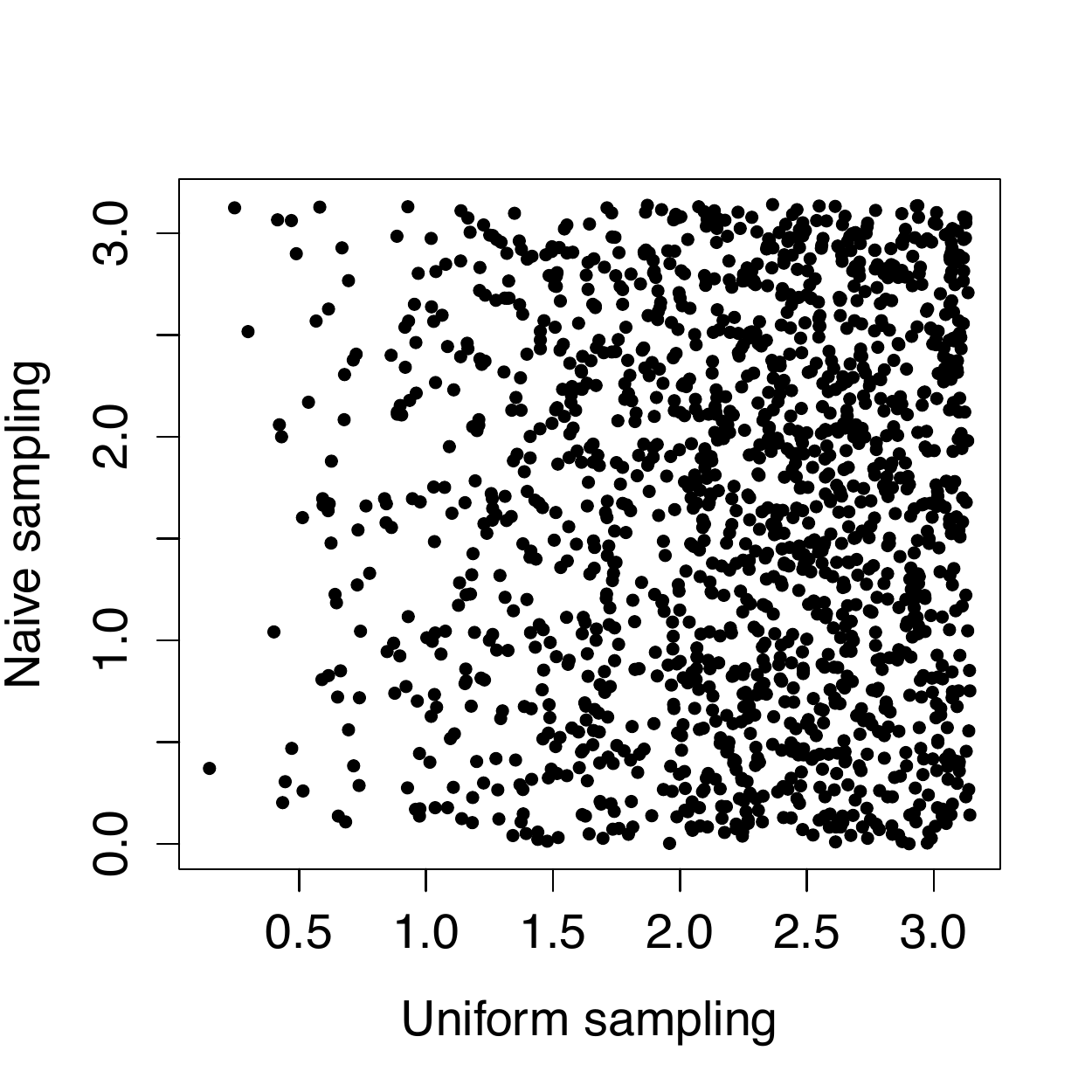}
    \caption*{}
  \end{minipage}
  \hfill
  \begin{minipage}[b]{0.62\textwidth}
  \centering
   \begin{subfigure}[b]{0.6\textwidth}
   \includegraphics[width=\textwidth]{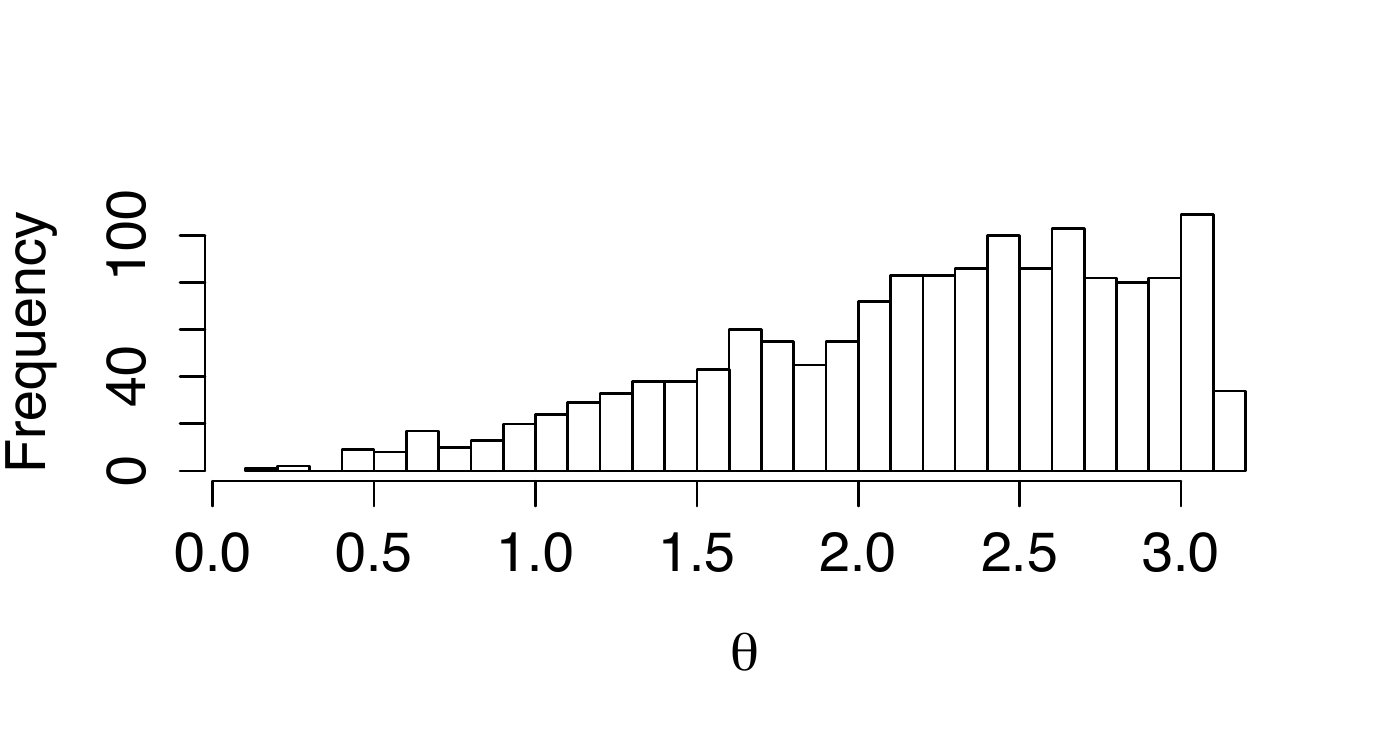}
   \end{subfigure}
   \begin{subfigure}[b]{0.6\textwidth}
   \includegraphics[width=\textwidth]{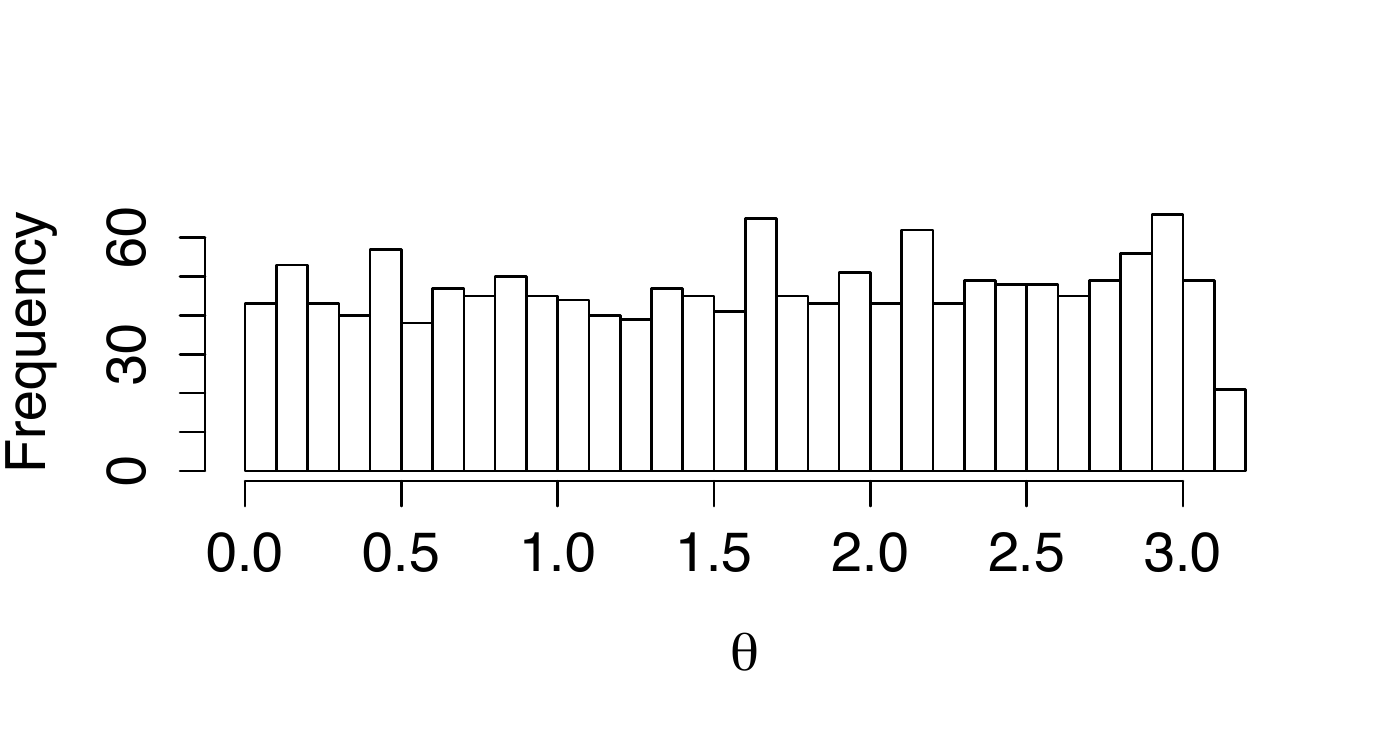}
  \end{subfigure}
  \end{minipage}
  \caption{Rotation angles. \textit{Left}: the scatter-plot of rotation angles for the uniform versus the naive sampler. \textit{Top right}: the histogram of rotation angles for the uniform sampler. \textit{Bottom right}: the histogram of rotation angles for the naive sampler. Each plot is based on $N = 1500$ samples.}
  \label{fig:hist}
\end{figure}
A more natural problem is to test if the full distributions on $\mli{SO}(3)$ are equal. The problem of goodness of fit for rotation data has been studied extensively, and there are various tests available. Two commonly used tests are Downs' generalization of the Rayleigh test \citep{downs1972orientation}  and
Prentice's generalization of Gin\'e's $G_n$ test \citep{prentice1978invariant, gine1975invariant}. For more details and references to the literature, see \citet[Section 13.2.2]{mardiadirectional}.
Given a sample $x_1,\ldots,x_n \in \mli{SO}(3)$, the generalized Rayleigh statistic is given by
\begin{align*}
T_R = 3N \Tr(\bar{X}^T\bar{X}), \; \; \; \bar{X} = \frac{1}{N} \sum_{i=1}^N x_i  .
\end{align*}
The value of $T_R$ for the naive sampler based on the sample of size $N=1500$ (used to generate figures) is $1379.452$ which corresponds to $p$-value of virtually $0$ based on $1000$ simulations of the test value under the uniform distribution. Note that $T_R$ is asymptotically distributed as a $\chi^2_9$ which yields an asymptotic $p$-value consistent with the simulation-based $p$-value. This is not surprising because the difference between the two distributions is very significant. In fact, only $N= 20$ observations yield an asymptotic $p$-value of less than $0.002$ and simulated $p$-value of $0.003$.

More generally, data from manifolds, including $\mli{SO}(3)$, appear naturally in many applications. The general problem of testing for uniformity on manifolds has been studied by various researchers. \citet{gine1975invariant}, in one of the most notable contributions in this area, introduced a method to construct non-parametric statistical tests of uniformity on compact Riemannian manifolds. These tests are called Sobolev tests, corresponding to each Sobolev norm $\|.\|_s$. The case $s=0$ is briefly sketched below. Let $M$ be a compact Riemannian manifold, $\mu$ the uniform measure on $M$, and $\Delta$ the Laplace-Beltrami operator (Laplacian) of $M$ acting on the space of Schwartz distributions by duality. Denote by $E_k$ the $k$-th invariant eigenspace of $\Delta$ with eigenvalue $\sigma_k$. Let $\{f_i\}_{i=0}^\infty$ be an orthonormal basis of eigenfunctions for the Laplacian. Consider the following hypothesis testing problem. Given observations $x_1,\ldots, x_n \in M$, drawn independently from a probability measure $\nu$, test whether $\nu = \mu$. Let $\nu_n$ be the empirical measure $\nu_n = \frac{1}{n}\sum \delta_{x_i}$. For a sequence of positive weights $\bm{\alpha} = (\alpha_1,\alpha_2,\ldots)$, with $\sup |\alpha_k \sigma_k^s|<\infty$ for some $s > \frac{1}{2}\dim M$, define the test statistic $T_n^{\bm{\alpha}}$ as
\begin{align*}
T_n^{\bm{\alpha}}(\bm{x}) = n \sum_{k=1}^\infty \alpha_k \sum_{f_i\in E_k} \left[\int_M f_i\, d(\nu_n(\bm{x})-\mu)\right]^2,
\end{align*}
i.e.\ a weighted sum of the Fourier coefficients of $\nu_n(\bm{x})-\mu$ with respect to the orthonormal system $\{f_i\}$ with weights depending only on the eigenspaces. See \cite{gine1975invariant} for more details, including statistical properties of the Sobolev tests.

Although the program introduced in \cite{gine1975invariant} constructs statistical tests with desirable statistical properties, there remains substantial work to be done to carry the program out for any particular example. Several authors have studied and derived Sobolev tests for different examples including circular and directional data, tests of symmetry, and unitary eigenvalues; see \cite{prentice1978invariant,wellner1979permutation,jupp1983sobolev,jupp1985sobolev,
hermans1985new,baringhaus1991testing,sengupta2001optimal,coram2003new}.

\cite{coram2003new} introduce statistical tests of 
correspondence between zeros of the Riemann zeta function and eigenvalues of random unitary 
matrices. They incorporate Gin\'e's program for the eigenvalue distribution, $\mu$, induced by the Haar measure on the unitary group $U(n)$.
The orthonormal basis for $\mc{L}^2(\mu)$ is given by the characters $\chi_\lambda$ of the unitary group. The weights are chosen as $\alpha_\lambda = z^{|\lambda|}$ for a parameter $0<z<1$.  Given observations $x_1,\ldots, x_N$ , the test statistic is given by
\begin{align*}
T_N^z &= \sum_{\lambda\neq 0} z^{|\lambda|} \left|\frac{1}{N}\sum_{i=1}^N \chi_\lambda(x_i)\right|^2\\
			&= \frac{1}{N^2} \sum_{i,j=1}^N \sum_{\lambda\neq 0} z^{|\lambda|}\chi_\lambda(x_i)\chi_\lambda(x_j).
\end{align*}
 To compute $T_N^z$, a closed form for the inner sum on the right hand side is needed. The Cauchy identity provides a closed form for 
 \begin{align}\label{Exp}
\sum_{\lambda\neq 0} z^{|\lambda|} \chi_\lambda(g)\chi_\lambda(h),
\end{align}
as given below in (\ref{eq:CauchyChar2}).

\cite{sepehri2016test} considers a similar goodness-of-fit testing problem on $\mli{SO}(2m+1)$. Let $\mu$ be the eigenvalue distribution induced by the Haar measure on $\mli{SO}(2m+1)$. Given data $\bm{x}_1,\ldots,\bm{x}_N$ drawn independently from a measure $\nu$, test if $\nu = \mu$ (note that each $\bm{x}_i$ is a vector of eigenvalues). Following the Gin\'e's program, the author considers the test based on
\begin{align*}
S^z_N &= \sum_{\lambda\ne 0} z^{|\lambda|} \left|\frac{1}{N} \sum_{i=1}^N so_\lambda(g_i)\right|^2,\\
			&= \frac{1}{N^2}\sum_{i,j=1}^N \sum_{\lambda\ne 0} z^{|\lambda|} so_\lambda(g_i) so_\lambda(g_j),
\end{align*}
where $so_\lambda$ is the character of $\mli{SO}(2m+1)$ corresponding to the partition $\lambda$, and $g_i \in \mli{SO}(2m+1)$ is a rotation with eigenvalue $\bm{x}_i$.

To use the test based on $S^z_N$, a closed form for the inner sum is needed. The Cauchy identity (\ref{eq:CauchySOodd}) for the orthogonal group is derived to address this problem. It suggests a tractable method to compute $S^z_N$. We refer the reader to \cite{sepehri2016test} for more details and statistical properties of the test based on $S^z_N$, as well as numerical results on synthetic and real data.

There is a sizable literature on generalizations of the Cauchy identity to larger classes of orthogonal polynomials. A major line of research focuses on generalizing the Cauchy identity to more general families of orthogonal polynomials, such as Hall-Littlewood polynomials \cite{macdonald1995symmetric,betea2016refined,warnaar2008bisymmetric,kirillov1999q, vuletic2009generalization}. Although the identities introduced in these works generalize the Cauchy identity to extended families of orthogonal polynomials, which specialize to Schur function, they do not offer any identity that results is a closed form for (\ref{Exp}). 

There is also substantial work concerning generalization of the Cauchy identity to \textit{generalized Schur functions}, which include the characters of the compact groups \cite{sundaram1986combinatorics,sundaram1990cauchy,sundaram1990orthogonal,fu2009non,brent2016discrete,okada1998applications,krattenthaler2001advanced}. The most related works among these to our case of interest seem to be \cite{sundaram1986combinatorics,sundaram1990cauchy,sundaram1990orthogonal}. They introduced Cauchy-type identities for the orthogonal and symplectic groups which present closed form expressions for 
\begin{align*}
\sum_{\lambda} s_\lambda (x) sp_\lambda (y^{\pm}) \quad \text{and} \quad \sum_{\lambda} s_\lambda (x) so_{2n+1,\lambda} (y^{\pm}) .
\end{align*}

All of the identities introduced in this part of the literature involves series where each term is the product of a Schur function and a generalized Schur function; whereas, what is needed in the application of interest in this paper is an identity that involves series in which each term is a product of two generalized Schur functions of the same type, as in (\ref{Exp}) .

None of the aforementioned work gives closed form expressions for (\ref{Exp}). However, it was brought to our attention that a closed form expression for a similar quantity for the symplectic group has recently appeared in an independent work \cite[Remark 7]{wheeler2016refined}. In particular, it proves our identity (\ref{eq:CauchySp}) concerning the symplectic group. The author was not aware of that work while writing this paper, and proof techniques are different. To the best of our knowledge, closed form expressions for (\ref{Exp}) have not appeared in the literature for the orthogonal groups of odd and even dimensions (Type B and D root systems). The main contribution of the current paper is to derive and prove such a generalization for other root systems.

\section{Background} 
The Cauchy identity is a well-known result about characters of the unitary group. 
For two sets of variables, $x_1,\ldots, x_n$ and $y_1,\ldots, y_m$, the Cauchy identity asserts that
\begin{align}\label{eq:CauchySchur}
\sum_{\lambda} s_\lambda(x_1,\ldots,x_n)s_\lambda(y_1,\ldots,y_m) = \prod_{i,j} \frac{1}{1-x_iy_j},
\end{align}
where the sum is over all partitions $\lambda$ of all non-negative integers and $s_\lambda$ is the Schur polynomial. The special case of $m=n$ is of particular interest here. Schur polynomials are related to characters of unitary groups in the following way: Let $g\in U(n)$ with eigenvalues $\alpha_1,\ldots,\alpha_n$ and $\lambda$ be a partition with at most $n$ parts. Then,
\begin{align*}
\chi_\lambda(g) = s_\lambda(\alpha_1,\ldots,\alpha_n),
\end{align*}
where $\chi_\lambda$ is the irreducible character corresponding to $\lambda$.
Thus, the representation theoretical version of the Cauchy identity can be written as
\begin{align}\label{eq:CauchyChar}
\sum_{\lambda} \chi_\lambda(g)\chi_\lambda(h) = \prod_{i=1}^n \prod_{j=1}^n  \frac{1}{1-\alpha_i \beta_j},
\end{align}
where $g\in U(n)$ has eigenvalues $\alpha_1,\ldots,\alpha_n$ and $ h\in U(n)$ has 
eigenvalues $\beta_1,\ldots,\beta_n$. Homogeneity of Schur polynomials yields

\begin{align}\label{eq:CauchyChar2}
\sum_{\lambda} z^{|\lambda|} \chi_\lambda(g)\chi_\lambda(h) = \prod_{i=1}^n \prod_{j=1}^n  \frac{1}{1-z \alpha_i \beta_j}
\end{align}
for any $0< z <1$.

The Cauchy identity (\ref{eq:CauchyChar}) has a conceptual interpretation and generalization from a representation theoretical point of view. This is sketched briefly following \citet[ch. 38]{bump2004lie}. Let $\text{Mat}_n(\mb{C})$ be the set of all $n\times n$ complex matrices and $\Sigma^\circ$, the ring of polynomials on $\text{Mat}_n(\mb{C})$. Define the action $\Pi^\circ$ of $U(n)\times U(n)$ on $\Sigma^\circ$ as follows. For $(g,h) \in U(n)\times U(n)$ and $f\in \Sigma^\circ$ define
\begin{align*}
\left(\Pi^\circ (g,h)f\right) (x) = f(g^t x h).
\end{align*}
Then, as a $U(n)\times U(n)$ representation, $\Pi^\circ$ decomposes into irreducible representations as
\begin{align}\label{CauchyCategorify}
\Pi^\circ \cong \sum_\lambda \pi_\lambda \otimes \pi_\lambda,
\end{align}
where $\lambda$ runs through all partitions of length $\le n$, and $\pi_\lambda$ is the corresponding irreducible representation of $U(n)$. Taking traces on both sides of (\ref{CauchyCategorify}) (in a proper sense of trace for operators on infinite dimensional spaces) yields the Cauchy identity (\ref{eq:CauchyChar}).
For details of this and an excellent textbook treatment of the Cauchy identity, see \citet[ch. 38]{bump2004lie}. 
A symmetric function theoretical point of view has been detailed in \citet[ch. 1]{macdonald1995symmetric}.

\section{Generalization of the Cauchy Identity}\label{sec:CauchyIdentity}
The identities for classical groups are presented based on the type of root systems. 

\subsection{Type B}
This section gives a Cauchy identity for the compact classical group of type B, i.e.\ special orthogonal group in odd dimensions. 
The result for the full orthogonal group, $O(2m+1)$, will follow from this. Before stating the result, 
recall that irreducible representations of $\mathit{SO}(2m+1)$ are labeled by partitions $\lambda$, of 
non-negative integers, with at most $m$ parts, see e.g.\ Proposition 3.1.20 in \citet[ch. 3]{goodman2009symmetry}. Any 
element $g\in \mathit{SO}(2m+1)$ has an eigenvalue equal to one, and the rest of eigenvalues come 
in conjugate pairs. The character $so_\lambda$ corresponding to partition $\lambda$ is given 
explicitly by the Weyl Character Formula \cite[ch. 7]{goodman2009symmetry} as follows:
\begin{align}
so_{\lambda}(g) = \frac{\det\left(x_{i}^{\lambda_j+m-j+\frac{1}{2}}-x_{i}^
{-(\lambda_j+m-j+\frac{1}{2})} \right)}{\det\left(x_{i}^{m-j+\frac{1}{2}}-x_{i}^{-(m-j+\frac{1}{2})} \right)},
\end{align}
 where $x_1,x_1^{-1},\ldots, x_m , x_m^{-1}, 1$ are eigenvalues of $g$.
We are ready to state our first theorem:
\begin{theorem}[\textsc{Cauchy identity for $\mathit{SO}(2m+1)$}]\label{thm:SOodd}
Let $so_{\lambda}$ be the character of $\mathit{SO}(2m+1)$ corresponding to the partition $\lambda$, and $g,h \in \mathit{SO}(2m+1)$ with eigenvalues equal to $x_1,x_1^{-1},\ldots,x_m,x_m^{-1},1$ and $y_1,y_1^{-1},\ldots,y_m,y_m^{-1},1$, respectively. Then,

\begin{align}\label{eq:CauchySOodd}
\sum_{\lambda}\; z^{|\lambda|} so_{\lambda}(g)so_{\lambda}(h) &= \frac{(1-z)^m \det\left(C\right)}{z^{\binom{m}{2}}\prod_{i<j} \left(y_i+y_i^{-1}-(y_j+y_j^{-1})\right) \prod_{i<j} \left(x_i+x_i^{-1}-(x_j+x_j^{-1})\right)}
\end{align}

where $\lambda$ runs over all partitions of non-negative integers with at most $m$ parts, $0< z <1$ is an arbitrary parameter and $C$ is an $m\times m$ matrix defined as
\begin{align*}
C_{ij} = \frac{(1+z)^2+z(x_i+x_i^{-1}+y_j+y_j^{-1})}{(1-z x_i y_j)(1-z x_i^{-1} y_j)(1-z x_i y_j^{-1})(1-z x_i^{-1} y_j^{-1})}.
\end{align*}
\end{theorem}
\begin{proof}
The proof builds on the classical Cauchy identity in \citet[pg. 34]{macdonald1995symmetric}.
Let $x_1,\ldots,x_m$ and $y_1,\ldots,y_m$ be two sets of variables. Consider the matrix $X$ defined as
\begin{align*}
X_{ij} = \sum_{k\ge 0} z^k (x_i^{k+\frac{1}{2}}-x_i^{-(k+\frac{1}{2})})(y_j^{k+\frac{1}{2}}-y_j^{-(k+\frac{1}{2})}).
\end{align*}
The $(i,j)$ entry, $X_{ij}$, can be simplified as follows

\begin{align*}
 X_{ij} &= \sum_k z^k (x_i^{k+\frac{1}{2}}-x_i^{-(k+\frac{1}{2})})(y_j^{k+\frac{1}{2}}-y_j^{-(k+\frac{1}{2})})\\
 			&= x_i^{\frac{1}{2}}y_j^{\frac{1}{2}}\sum_k z^k x_i^k y_j^k - x_i^{\frac{-1}{2}}y_j^{\frac{1}{2}}\sum_k z^k x_i^{-k} y_j^k-x_i^{\frac{1}{2}}y_j^{\frac{-1}{2}}\sum_k z^k x_i^k y_j^{-k}+x_i^{\frac{-1}{2}}y_j^{\frac{-1}{2}}\sum_k z^k x_i^{-k} y_j^{-k}\\
 			&= \frac{x_i^{\frac{1}{2}}y_j^{\frac{1}{2}}}{1-z x_i y_j}-\frac{x_i^{\frac{-1}{2}}y_j^{\frac{1}{2}}}{1-z x_i^{-1} y_j}-\frac{x_i^{\frac{1}{2}}y_j^{\frac{-1}{2}}}{1-z x_i y_j^{-1}}+\frac{x_i^{\frac{-1}{2}}y_j^{\frac{-1}{2}}}{1-z x_i^{-1} y_j^{-1}},
\end{align*} 

which simplifies to
\begin{align}\label{eq:EntriesX}
X_{ij}= \frac{(x_i^{\frac{1}{2}}-x_i^{\frac{-1}{2}})(y_j^{\frac{1}{2}}-y_j^{\frac{-1}{2}})
(1-z)[(1+z)^2+z(x_i+x_i^{-1}+y_j+y_j^{-1})]}{(1-z x_i y_j)(1-z x_i^{-1} y_j)(1-z x_i y_j^{-1})(1-z x_i^{-1} y_j^{-1})}.
\end{align}
This shows the relation between $X$ and the determinant on the right hand side of identity 
(\ref{eq:CauchySOodd}). Now expand $\det(X)$ as follows
\begin{align*}
\det(X) &= \det\left( \sum_{k\ge 0} z^k (x_i^{k+\frac{1}{2}}-x_i^{-(k+\frac{1}{2})})
(y_j^{k+\frac{1}{2}}-y_j^{-(k+\frac{1}{2})}) \right)\\
           &= \sum_{\pi \in S_m} sign(\pi) \prod_{i\le m}\left(\sum_{k\ge 0}
            z^k (x_i^{k+\frac{1}{2}}-x_i^{-(k+\frac{1}{2})})(y_{\pi(i)}^{k+\frac{1}{2}}-y_{\pi(i)}^{-(k+\frac{1}{2})})\right).        
\end{align*}
Expanding the product on the right hand side leads to
\begin{align*}
\det(X)  &= \sum_{\pi \in S_m} sign(\pi) \sum_{a_1,a_2,\ldots,a_m\ge 0} \; \prod_{i\le m}z^{a_i} 
\left((x_i^{a_i+\frac{1}{2}}-x_i^{-(a_i+\frac{1}{2})})(y_{\pi(i)}^{a_i+\frac{1}{2}}-y_{\pi(i)}^{-(a_i+\frac{1}{2})})\right).
\end{align*}
Changing the order of summation and using the definition of $m\times m$ determinants yields

\begin{align*}
\det(X) &=  \sum_{a_1,a_2,\ldots,a_m\ge 0} \; z^{\sum_i a_i} \sum_{\pi \in S_m} sign(\pi) 
\prod_{i\le m}\left((x_i^{a_i+\frac{1}{2}}-x_i^{-(a_i+\frac{1}{2})})(y_{\pi(i)}^{a_i+\frac{1}{2}}-y_{\pi(i)}^{-(a_i+\frac{1}{2})})\right)\\
         &= \sum_{a \in\Z_+^m} z^{|a|} \det\left((x_i^{a_i+\frac{1}{2}}-x_i^{-(a_i+\frac{1}{2})})(y_{j}^{a_i+\frac{1}{2}}-y_{j}^{-(a_i+\frac{1}{2})})\right)\\
         &= \sum_{a \in \Z_+^m} z^{|a|} \det\left(y_{j}^{a_i+\frac{1}{2}}-y_{j}^{-(a_i+\frac{1}{2})}\right) \prod_i (x_i^{a_i+\frac{1}{2}}-x_i^{-(a_i+\frac{1}{2})}).
\end{align*}

The last equality can be rewritten as follows by ordering the $m$-tuple $a\in \Z^m$.

\begin{align*}
\det(X) &= \sum_{a_1 \ge \ldots \ge a_m\ge 0}\; \sum_{\sigma\in S_m} z^{|a|} 
\det\left(y_{j}^{a_i+\frac{1}{2}}-y_{j}^{-(a_i+\frac{1}{2})}\right) sign(\sigma) 
\prod_i (x_i^{a_{\sigma(i)}+\frac{1}{2}}-x_i^{-(a_{\sigma(i)}+\frac{1}{2})})\\
          &= \sum_{a_1 \ge \ldots \ge a_m\ge 0}\;  z^{|a|} \det\left(y_{j}^{a_i+\frac{1}{2}}-y_{j}^{-(a_i+\frac{1}{2})}\right)
           \sum_{\sigma\in S_m}sign(\sigma) \prod_i (x_i^{a_{\sigma(i)}+\frac{1}{2}}-x_i^{-(a_{\sigma(i)}+\frac{1}{2})})\\
          &= \sum_{a_1 \ge \ldots \ge a_m\ge 0}\; z^{|a|} \det\left(y_{j}^{a_i+\frac{1}{2}}-y_{j}^{-(a_i+\frac{1}{2})}\right) 
          \det\left(x_{j}^{a_i+\frac{1}{2}}-x_{j}^{-(a_i+\frac{1}{2})}\right).
\end{align*}

Note that if two of $a_i$'s are equal then both determinants on the right hand side are zero. Thus, it may be 
assumed that $a_1 > a_2> \ldots > a_m \ge 0$. This is equivalent to the condition that $a_i -(m-i)$ 
is a non-increasing sequence. Define $\lambda_i = a_i -m +i$, then $\lambda_1 \ge \ldots \ge \lambda_m\ge 0$. 
The last equality translates to the following in terms of $\lambda$:
\begin{align*}
\det(X) &= \sum_{\lambda_1 \ge \ldots \ge \lambda_m \ge 0}\; z^{|\lambda|+\binom{m}{2}} \det\left(A^\lambda\right) \det\left(B^\lambda\right),\\
A_{i,j}^\lambda &=y_{j}^{\lambda_i+m-i+\frac{1}{2}}-y_{j}^{-(\lambda_i+m-i+\frac{1}{2})},\\
B_{i,j}^\lambda &= x_{j}^{\lambda_i+m-i+\frac{1}{2}}-x_{j}^{-(\lambda_i+m-i+\frac{1}{2})}.
\end{align*}
Assume  $x_1,\ldots,x_m$ and $y_1,\ldots,y_m$ are such that there exist $g,h \in \mathit{SO}(2m+1)$ with 
eigenvalues equal to $x_1,x_1^{-1},\ldots,x_m,x_m^{-1},1$ and  $y_1,y_1^{-1},\ldots,y_m,y_m^{-1},1$, 
respectively. Dividing by a factor of the following form
\begin{align*}
z^{\binom{m}{2}} \det\left(y_{j}^{m-i+\frac{1}{2}}-y_{j}^{-(m-i+\frac{1}{2})}\right) \det\left(x_{j}^{m-i+\frac{1}{2}}-x_{j}^{-(m-i+\frac{1}{2})}\right),
\end{align*}
and using the Weyl character formula yields
\begin{align}\label{eq:SOCauchyIntermidiate}
\sum_{\lambda}\; z^{|\lambda|} so_{\lambda}(g)so_{\lambda}(h) =
\frac{z^{-\binom{m}{2}}\det(X)}{ \det\left(y_{j}^{m-i+\frac{1}{2}}-y_{j}^{-(m-i+\frac{1}{2})}\right) 
\det\left(x_{j}^{m-i+\frac{1}{2}}-x_{j}^{-(m-i+\frac{1}{2})}\right)},
\end{align}
where $\lambda$ runs over all partitions of an arbitrary integer with at most $m$ parts. To finish the proof, the right side of (\ref{eq:SOCauchyIntermidiate}) is simplified; from (\ref{eq:EntriesX}),
\begin{align*}
\det(X) &= \det\left( \frac{(x_i^{\frac{1}{2}}-x_i^{\frac{-1}{2}})(y_j^{\frac{1}{2}}-y_j^{\frac{-1}{2}})
(1-z)[(1+z)^2+z(x_i+x_i^{-1}+y_j+y_j^{-1})]}{(1-z x_i y_j)(1-z x_i^{-1} y_j)(1-z x_i y_j^{-1})(1-z x_i^{-1} y_j^{-1})}\right)\\
			&= (1-z)^m \prod_{i} (x_i^{\frac{1}{2}}-x_i^{\frac{-1}{2}})(y_i^{\frac{1}{2}}-y_i^{\frac{-1}{2}})
			\det\left( C\right),
\end{align*}
where $C$ is defined in the statement of the theorem. To get to the final form, one more simplification is needed. Note that
\begin{align*}
x_j^{m-i+\frac{1}{2}}-x_j^{-(m-i+\frac{1}{2})} = (x_j^{\frac{1}{2}}-x_j^{\frac{-1}{2}}) (x_j^{m-i}+x_j^{m-i-1}+\ldots +x_j^{-(m-i)})
\end{align*}
leads to 
\begin{align*}
\det\left(x_{j}^{m-i+\frac{1}{2}}-x_{j}^{-(m-i+\frac{1}{2})}\right) &= \det\left( x_j^{m-i}+x_j^{m-i-1}+\ldots +x_j^{-(m-i)}\right) \prod_{1\le j\le m} (x_j^{\frac{1}{2}}-x_j^{\frac{-1}{2}})\; .
\end{align*}
Subtracting $(i+1)^{th}$ row from the $i^{th}$ row for $i=1,2,\ldots,m-1$ yields
\begin{align*}
\det\left(x_{j}^{m-i+\frac{1}{2}}-x_{j}^{-(m-i+\frac{1}{2})}\right) &= \det\left( x_j^{m-i}+x_j^{-(m-i)}\right) \prod_{1\le j\le m} (x_j^{\frac{1}{2}}-x_j^{\frac{-1}{2}})\\
																										   &= \det\left( (x_j+x_j^{-1})^{m-i}\right) \prod_{1\le j\le m} (x_j^{\frac{1}{2}}-x_j^{\frac{-1}{2}})\\
																										   &= \prod_{i<j} (x_i+x_i^{-1}-(x_j+x_j^{-1})) \prod_{1\le j\le m} (x_j^{\frac{1}{2}}-x_j^{\frac{-1}{2}}),
\end{align*}
where the second equality can be proved by simple row operations, and the last equality is the Vandermonde identity. Simplification of the determinant involving $y$ is identical to the calculations above. Substituting these simplifications in (\ref{eq:SOCauchyIntermidiate}) proves the theorem.
\end{proof}
\remark It is noteworthy that unlike the unitary group, the characters $o_\lambda$ 
of the orthogonal group are not homogeneous functions of the eigenvalues. 
This perhaps is the reason that the Cauchy identity is much more complicated for the orthogonal group than for the unitary group.

 With Theorem \ref{thm:SOodd} in hand, a Cauchy identity for the orthogonal group $O(2m+1)$ is a straightforward corollary . 
 \begin{cor}
 [\textsc{Cauchy identity for $O(2m+1)$}]\label{thm:CauchyOodd}
Let $o_{\lambda}$ be the character of $O(2m+1)$ corresponding to the partition $\lambda$, and $g,h \in O(2m+1)$ with eigenvalues equal to $x_1,x_1^{-1},\ldots,x_m,x_m^{-1},det(g)$ and $y_1,y_1^{-1},\ldots,y_m,y_m^{-1},det(h)$, respectively, then
\begin{align}\label{eq:CauchyOodd}
\sum_{\lambda}\; z^{|\lambda|} o_{\lambda}(g)o_{\lambda}(h) &= \frac{(1-det(gh)z)^m det\left( \frac{(1+det(gh)z)^2+z(det(h)(x_i+x_i^{-1})+det(g)(y_j+y_j^{-1}))}{(1-z x_i y_j)(1-z x_i^{-1} y_j)(1-z x_i y_j^{-1})(1-z x_i^{-1} y_j^{-1})}\right)}{z^{\binom{m}{2}}\prod_{i<j} \left(y_i+y_i^{-1}-(y_j+y_j^{-1})\right) \prod_{i<j} \left(x_i+x_i^{-1}-(x_j+x_j^{-1})\right)},
\end{align}
where $\lambda$ runs over all partitions of arbitrary integers with at most $m$ parts.
\end{cor}
\begin{proof}
Let $n=2m+1$. Recall that the irreducible representations of $O(n)$ are 
labeled by partitions $\lambda$ of arbitrary non-negative integers for which 
$\lambda_1^\prime+\lambda_2^\prime \le n$, where $\lambda^\prime$ is the transpose of $\lambda$ 
defined as $\lambda^{\prime}_i = \# \{j: \lambda_j \ge i\}$ (see \citet[ch. 10]{goodman2009symmetry}). 
For a partition $\lambda$ with $\lambda^{\prime}_1\le m$, define $\tilde{\lambda}$ as the 
partition with $\tilde{\lambda}^{\prime} = (n-\lambda^{\prime}_1,\lambda^{\prime}_2,..., \lambda^{\prime}_l)$. 
For each partition $\lambda$ with $\lambda^{\prime}_1\le m$, there exist an irreducible 
representation of $O(n)$ labeled by $\lambda$ and one labeled with $\tilde{\lambda}$. Moreover, the Weyl character formula asserts that
\begin{align*}
o_{\lambda}(g) &= \frac{\det\left(x_{i}^{\lambda_j+m-j+\frac{1}{2}}-det(g) x_{i}^{-(\lambda_j+m-j+\frac{1}{2})} \right)}{\det\left(x_{i}^{m-j+\frac{1}{2}}- det(g)x_{i}^{-(m-j+\frac{1}{2})} \right)},\\
o_{\tilde{\lambda}}(g) &= det(g) o_{\lambda}(g).
\end{align*}
The last equality justifies summing only over $\lambda$ with at most $m$ parts. 
The rest of the proof is identical to that in Theorem \ref{thm:SOodd}.
\end{proof}
\subsubsection{Limit as $z\rightarrow 1$} The Identity (\ref{eq:CauchySOodd}) exhibits interesting limiting 
behavior as $z$ tends to $1$. Assume throughout this section that $g$ and $h$ share no common eigenvalues.
The right hand side is equal to zero at $z=1$, 
but the left hand side does not converge at $z=1$, because the series defining $X_{i,j}$ do not converge. 
However, under an extended notion of series convergence, i.e.\ Cesaro $C_\alpha$-summability, 
the left hand side converges to a limit at $z =1$ which coincides with the value of the right hand side. 
\begin{definition}[$C_k$-summability] Given a sequence $\{a_n\}$, define
\begin{align*}
A_n^{(-1)} = a_n,\;\;\; A_n^{(k)} = \sum_{i=1}^n A_i^{(k-1)} \;\;\; k=1,2,\ldots\; .
\end{align*}
Let $E_n^{(\alpha)}$ be the corresponding sequence $A_n^{(\alpha)}$ starting from the initial sequence $e_1=1, e_i =0$ for $i>1$ . The series $\sum_n a_n$ is called $C_\alpha$-summable to the value $s$ if 
\begin{align*}
\lim _{n\rightarrow \infty}\frac{A_n^{(\alpha)}}{E_n^{(\alpha)}} = s .
\end{align*}
\end{definition}
It is known that the geometric series $1+z+z^2+\ldots$ is $C_1$-summable for $|z|=1,\; z\ne 1$ with $C_1$-limit $\frac{1}{1-z}$.
This ensures that the entries $X_{i,j}$ of the matrix $X$ are $C_1$-summable to their corresponding closed forms given in (\ref{eq:EntriesX}). Then, the following lemma implies that all the terms in the expansion of $\det (X)$ are $C_ {2m-1}$-summable. Therefore, the left hand side of (\ref{eq:CauchySOodd}) is $C_{2m-1}$-summable at $z=1$.
\begin{lemma}[Theorem 277 in \cite{knopp1948theory}]\label{lemma1} If $\sum a_n$ is $C_\alpha$-summable to the value $A$ and $\sum b_n$ is $C_\beta$-summable to the value $B$, then their Cauchy product, 
\begin{align*}
\sum c_n = \sum (a_0b_n+a_1b_{n-1}+\ldots+a_nb_0),
\end{align*}
is certainly $C_\gamma$-summable to the value $C = AB$, where $\gamma = \alpha+\beta +1$.
\end{lemma}
The following lemma justifies taking the limit $z\rightarrow 1$ on both sides of (\ref{eq:CauchySOodd}). This, together with the discussion above, imply that, for $z \rightarrow 1$, the left hand side converges to the value of the right hand side which is equal to zero.
\begin{lemma}[Theorem 278 in \cite{knopp1948theory}]\label{lemma2} If a power series $f(z) = \sum a_n z^n$ has convergence radius $1$ and is $C_k$-summable to the value $s$ at the point $z = 1$, then
\begin{align*}
f(z) \rightarrow s
\end{align*}
for every mode of approach of $z$ to $1$, in which $z$ remains within an angle of vertex $+1$, bounded by two fixed chords of the unit circle. This certainly includes the case $z\rightarrow 1$ for $z \in [0,1]$.
\end{lemma}
\cite{macphail1941cesaro} proved that the series $\sum_n P(n) e^{i n\theta}$ is $C_k$-summable for any polynomial $P$ of degree less than $k$. This allows differentiating both sides of identity (\ref{eq:CauchySOodd}) with respect to $z$. More precisely, write identity (\ref{eq:CauchySOodd}) as
\begin{align*}
\sum_{\lambda}\; z^{\binom{m}{2}+|\lambda|} so_{\lambda}(g)so_{\lambda}(h) &= \frac{(1-z)^m \det\left(C\right)}{\prod_{i<j} \left(y_i+y_i^{-1}-(y_j+y_j^{-1})\right) \prod_{i<j} \left(x_i+x_i^{-1}-(x_j+x_j^{-1})\right)}.
\end{align*}
Differentiating $m$ times with respect to $z$ yields
\begin{align*}
\sum_{\lambda}\; \frac{\left(\binom{m}{2}+|\lambda|\right)!}{\left(\binom{m}{2}+|\lambda|-m\right)!} z^{\binom{m}{2}+|\lambda|-m}so_{\lambda}(g)so_{\lambda}(h) &= \frac{\frac{\partial^m}{\partial z^m}\left[(1-z)^m \det\left(C\right)\right]}{\prod_{i<j} \left(y_i+y_i^{-1}-(y_j+y_j^{-1})\right) \prod_{i<j} \left(x_i+x_i^{-1}-(x_j+x_j^{-1})\right)}.
\end{align*}
Note that the left hand side is Cesaro summable at $z=1$ to the value of the right hand side at $z=1$. This follows from the Leibniz rule, Cesaro summablity of the derivatives of geometric series, and lemma \ref{lemma1} using an argument similar to the one above. Then, applying lemma \ref{lemma2}, and simplification of the right hand side, yields
\begin{align} \label{eq:Z1limit}
 \lim_{z\rightarrow 1}\; \sum_{\lambda}\; \binom{\binom{m}{2}+|\lambda|}{m}\, z^{\binom{m}{2}+|\lambda|-m}so_{\lambda}(g)so_{\lambda}(h) &=\frac{(-1)^m\det\left( \frac{u_i+v_j}{(u_i-v_j)^2}\right)}{\prod_{i<j} \left(u_i-u_j\right) \prod_{i<j} \left(v_i-v_j\right)},
\end{align}
where $u_i = x_i+x_i^{-1}+2$ and $v_j = y_j+y_j^{-1}+2$, and the limit is taken over real number $z\in (0,1)$ approaching $1$. This suggests using the right hand side of (\ref{eq:Z1limit}) for computational purposes.
\subsection{Type C}
The type C root system corresponds to the compact (real) \textit{symplectic} group. 
The statement and proof of the type C Cauchy identity is very similar to that of type B. Irreducible 
representations of the symplectic group of order $m$, $\mli{Sp}(2m)$, are labeled by partitions $\lambda$, of 
non-negative integers, with at most $m$ parts (see Proposition 3.1.20 in \citet[ch. 3]{goodman2009symmetry}). Let $g\in \mli{Sp}(2m)$ have eigenvalues  $x_1,x_1^{-1},\ldots, x_m , x_m^{-1}$. The character $sp_\lambda$ corresponding to partition $\lambda$ is given explicitly by the Weyl Character Formula:
\begin{align}
sp_{\lambda}(g) = \frac{\det\left(x_{i}^{\lambda_j+m-j+1}-x_{i}^
{-(\lambda_j+m-j+1)} \right)}{\det\left(x_{i}^{m-j+1}-x_{i}^{-(m-j+1)} \right)}.
\end{align}
With this notation, the main result of this section is:
\begin{theorem}[\textsc{Cauchy identity for $\mathit{Sp}(2m)$}]\label{thm:Sp}
Let $sp_{\lambda}$ be the character of $\mathit{Sp}(2m)$ corresponding to the partition $\lambda$, and $g,h \in \mathit{Sp}(2m)$ with eigenvalues equal to $x_1,x_1^{-1},\ldots,x_m,x_m^{-1}$ and $y_1,y_1^{-1},\ldots,y_m,y_m^{-1}$, respectively. Then,

\begin{align}\label{eq:CauchySp}
\sum_{\lambda}\; z^{|\lambda|} sp_{\lambda}(g)\;sp_{\lambda}(h) &= \frac{(1-z^2)^m \det\left( \frac{1}{(1-z x_i y_j)(1-z x_i^{-1} y_j)(1-z x_i y_j^{-1})(1-z x_i^{-1} y_j^{-1})}\right)}{z^{\binom{m}{2}}\prod_{i<j} \left(y_i+y_i^{-1}-(y_j+y_j^{-1})\right) \prod_{i<j} \left(x_i+x_i^{-1}-(x_j+x_j^{-1})\right)}
\end{align}

where $\lambda$ runs over all partitions of non-negative integers with at most $m$ parts and $0< z <1$ is an arbitrary parameter.
\end{theorem}
\begin{proof}
Let $x_1,\ldots,x_m$ and $y_1,\ldots,y_m$ be two sets of variables. Consider the matrix $X$ defined as
\begin{align*}
X_{ij} = \sum_{k\ge 0} z^k (x_i^{k+1}-x_i^{-(k+1)})(y_j^{k+1}-y_j^{-(k+1)}).
\end{align*}
Write
\begin{align}\label{eq:EntriesXSp}
X_{ij}= \frac{(1-z^2)(x_i-x_i^{-1})(y_j-y_j^{-1})}{(1-z x_i y_j)(1-z x_i^{-1} y_j)(1-z x_i y_j^{-1})(1-z x_i^{-1} y_j^{-1})}.
\end{align}
Expand $\det(X)$ as follows
\begin{align*}
\det(X) &= \det\left( \sum_{k\ge 0} z^k (x_i^{k+1}-x_i^{-(k+1)})
(y_j^{k+1}-y_j^{-(k+1)}) \right)\\
           &= \sum_{\pi \in S_m} sign(\pi) \prod_{i\le m}\left(\sum_{k\ge 0}
            z^k (x_i^{k+1}-x_i^{-(k+1)})(y_{\pi(i)}^{k+1}-y_{\pi(i)}^{-(k+1)})\right).        
\end{align*}
Expanding the product on the right hand side leads to
\begin{align*}
\det(X)  &= \sum_{\pi \in S_m} sign(\pi) \sum_{a_1,a_2,\ldots,a_m\ge 0} \; \prod_{i\le m}z^{a_i} 
\left((x_i^{a_i+1}-x_i^{-(a_i+1)})(y_{\pi(i)}^{a_i+1}-y_{\pi(i)}^{-(a_i+1)})\right).
\end{align*}
Changing the order of summation and using the definition of $m\times m$ determinant yields

\begin{align*}
\det(X) &=  \sum_{a_1,a_2,\ldots,a_m\ge 0} \; z^{\sum_i a_i} \sum_{\pi \in S_m} sign(\pi) 
\prod_{i\le m}\left((x_i^{a_i+1}-x_i^{-(a_i+1)})(y_{\pi(i)}^{a_i+1}-y_{\pi(i)}^{-(a_i+1)})\right)\\
         &= \sum_{a \in \Z_+^m} z^{|a|} \det\left(y_{j}^{a_i+1}-y_{j}^{-(a_i+1)}\right) \prod_i (x_i^{a_i+1}-x_i^{-(a_i+1)}).
\end{align*}

Rewrite the last equality as follows by ordering the $m$-tuple $a\in \Z^m$.

\begin{align*}
\det(X) &= \sum_{a_1 \ge \ldots \ge a_m\ge 0}\; \sum_{\sigma\in S_m} z^{|a|} 
\det\left(y_{j}^{a_i+1}-y_{j}^{-(a_i+1)}\right) sign(\sigma) 
\prod_i (x_i^{a_{\sigma(i)}+1}-x_i^{-(a_{\sigma(i)}+1)})\\
              &= \sum_{a_1 \ge \ldots \ge a_m\ge 0}\; z^{|a|} \det\left(y_{j}^{a_i+1}-y_{j}^{-(a_i+1)}\right) 
          \det\left(x_{j}^{a_i+1}-x_{j}^{-(a_i+1)}\right).
\end{align*}

If two of $a_i$'s are equal then both determinants on RHS are zero. Thus, assume $a_1 > a_2> \ldots > a_m \ge 0$. Define $\lambda_i = a_i -m +i$, then $\lambda_1 \ge \ldots \ge \lambda_m\ge 0$. 
The last equality translates to the following in terms of $\lambda$:
\begin{align*}
\det(X) &= \sum_{\lambda_1 \ge \ldots \ge \lambda_m \ge 0}\; z^{|\lambda|+\binom{m}{2}} \det\left(A^\lambda\right) \det\left(B^\lambda\right),\\
A_{i,j}^\lambda &=y_{j}^{\lambda_i+m-i+1}-y_{j}^{-(\lambda_i+m-i+1)},\\
B_{i,j}^\lambda &= x_{j}^{\lambda_i+m-i+1}-x_{j}^{-(\lambda_i+m-i+1)}.
\end{align*}
Assume  $x_1,\ldots,x_m$ and $y_1,\ldots,y_m$ are such that there exist $g,h \in \mathit{Sp}(2m)$ with 
eigenvalues equal to $x_1,x_1^{-1},\ldots,x_m,x_m^{-1}$ and  $y_1,y_1^{-1},\ldots,y_m,y_m^{-1}$, 
respectively. Dividing by a factor of the following form
\begin{align*}
z^{\binom{m}{2}} \det\left(y_{j}^{m-i+1}-y_{j}^{-(m-i+1)}\right) \det\left(x_{j}^{m-i+1}-x_{j}^{-(m-i+1)}\right),
\end{align*}
and using the Weyl character formula yields
\begin{align}\label{eq:SpCauchyIntermidiate}
\sum_{\lambda}\; z^{|\lambda|} sp_{\lambda}(g)sp_{\lambda}(h) =
\frac{z^{-\binom{m}{2}}\det(X)}{ \det\left(y_{j}^{m-i+1}-y_{j}^{-(m-i+1)}\right) 
\det\left(x_{j}^{m-i+1}-x_{j}^{-(m-i+1)}\right)},
\end{align}
where $\lambda$ runs over all partitions of an arbitrary integer with at most $m$ parts. Using equation (\ref{eq:EntriesXSp})
\begin{align*}
\det(X) &= \det\left( \frac{(x_i-x_i^{-1})(y_j-y_j^{-1})
(1-z^2)}{(1-z x_i y_j)(1-z x_i^{-1} y_j)(1-z x_i y_j^{-1})(1-z x_i^{-1} y_j^{-1})}\right)\\
			&= (1-z^2)^m \prod_{i} (x_i-x_i^{-1})(y_i-y_i^{-1})
			\det\left(  \frac{1}{(1-z x_i y_j)(1-z x_i^{-1} y_j)(1-z x_i y_j^{-1})(1-z x_i^{-1} y_j^{-1})}\right).
\end{align*}
To get to the final form, one more simplification is needed. Note that
\begin{align*}
x_j^{m-i+1}-x_j^{-(m-i+1)} = (x_j-x_j^{-1}) (x_j^{m-i}+x_j^{m-i-2}+\ldots +x_j^{-(m-i)})
\end{align*}
leads to 
\begin{align*}
\det\left(x_{j}^{m-i+1}-x_{j}^{-(m-i+1)}\right) &= \det\left( x_j^{m-i}+x_j^{m-i-2}+\ldots +x_j^{-(m-i)}\right) \prod_{1\le j\le m} (x_j-x_j^{-1})\; .
\end{align*}
Subtracting $(i+2)^{th}$ row from the $i^{th}$ row for $i=1,2,\ldots,m-2$ yields
\begin{align*}
\det\left(x_{j}^{m-i+1}-x_{j}^{-(m-i+1)}\right) &= \det\left( x_j^{m-i}+x_j^{-(m-i)}\right) \prod_{1\le j\le m} (x_j-x_j^{-1})\\
																										   &= \det\left( (x_j+x_j^{-1})^{m-i}\right) \prod_{1\le j\le m} (x_j-x_j^{-1})\\
																										   &= \prod_{i<j} (x_i+x_i^{-1}-(x_j+x_j^{-1})) \prod_{1\le j\le m} (x_j-x_j^{-1}).
\end{align*}
 Substituting these simplifications in (\ref{eq:SpCauchyIntermidiate}) proves the theorem.
\end{proof}
\subsection{Type D}
The situation is more subtle for type D, which corresponds to the special orthogonal groups in even dimensions. Focus on $\mli{SO}(2m)$. The irreducible representations of $\mli{SO}(2m)$ are indexed by sequences of integers $\lambda =
(\lambda_1,\lambda_2,...,\lambda_m)$ and $\lambda^- = (\lambda_1,\lambda_2,...,\lambda_{m-1},- \lambda_m)$, where $\lambda$ is a partition with at most $m$ parts. The Weyl character formula is as follows. Let $x_1,x_1^{-1},\ldots, x_m , x_m^{-1}$ are eigenvalues of $g\in \mli{SO}(2m)$. If $\tilde{\lambda}_1 <m$ , i.e.\ $\lambda_m = 0 $, then
\begin{align}\label{eqn:WeylSOeven1}
so_{\lambda}(g) = \frac{\det\left(x_{i}^{\lambda_j+m-j}+x_{i}^
{-(\lambda_j+m-j)} \right)}{\det\left(x_{i}^{m-j}+x_{i}^{-(m-j)} \right)}.
\end{align}
If $\tilde{\lambda}_1 = m$, i.e.\ $\lambda_m > 0 $, then the character of the irreducible representation corresponding to $\lambda$ is given by
\begin{align}\label{eqn:WeylSOeven2}
so_{\lambda}(g) = \frac{\det\left(x_{i}^{\lambda_j+m-j}+x_{i}^
{-(\lambda_j+m-j)} \right)-\det\left(x_{i}^{\lambda_j+m-j}-x_{i}^
{-(\lambda_j+m-j)} \right)}{2\det\left(x_{i}^{m-j}+x_{i}^{-(m-j)} \right)},
\end{align}
and the one corresponding to $\lambda^-$ is given as
\begin{align}\label{eqn:WeylSOeven3}
so_{\lambda^-}(g) = \frac{\det\left(x_{i}^{\lambda_j+m-j}+x_{i}^
{-(\lambda_j+m-j)} \right)+\det\left(x_{i}^{\lambda_j+m-j}-x_{i}^
{-(\lambda_j+m-j)} \right)}{2\det\left(x_{i}^{m-j}+x_{i}^{-(m-j)} \right)}.
\end{align}
Define $\chi_\lambda$ as $so_\lambda$ if $\tilde{\lambda}_1 <m$, and as $so_\lambda+so_{\lambda^-}$ if $\tilde{\lambda}_1 = m$. Note that
$\chi_\lambda$ is the restriction of the irreducible character of $O(2m)$ to $\mli{SO}(2m)$; it is an irreducible character of $\mli{SO}(2m)$ if and only if $\tilde{\lambda}_1 <m$. We will state the Cauchy identity in terms of $\{\chi_\lambda\}$.
\begin{theorem}[\textsc{Cauchy identity for $\mathit{SO}(2m)$}]\label{thm:SOeven}
Let $\chi_{\lambda}$ be as above, and $g,h \in \mathit{SO}(2m)$ with eigenvalues equal to $x_1,x_1^{-1},\ldots,x_m,x_m^{-1}$ and $y_1,y_1^{-1},\ldots,y_m,y_m^{-1}$, respectively. Then,

\begin{align}\label{eq:CauchySOeven}
\sum_{\lambda}\; z^{|\lambda|} \chi_{\lambda}(g)\chi_{\lambda}(h) &= \frac{\det\left(\frac{1}{1-z x_i y_j}+\frac{1}{1-z x_i^{-1} y_j}+\frac{1}{1-z x_i y_j^{-1}}+\frac{1}{1-z x_i^{-1} y_j^{-1}}\right)}{z^{\binom{m}{2}}\prod_{i<j} \left(y_i+y_i^{-1}-(y_j+y_j^{-1})\right) \prod_{i<j} \left(x_i+x_i^{-1}-(x_j+x_j^{-1})\right)}
\end{align}

where $\lambda$ runs over all partitions of non-negative integers with at most $m$ parts and $0< z <1$ is an arbitrary parameter.
\end{theorem}%
\begin{proof}
Let $x_1,\ldots,x_m$ and $y_1,\ldots,y_m$ be two sets of variables. Consider the matrix $X$ defined as
\begin{align*}
X_{ij} &= \sum_{k\ge 0} z^k (x_i^{k }+x_i^{-k })(y_j^{k }+y_j^{-k })\\
			&= \frac{1}{1-z x_i y_j}+\frac{1}{1-z x_i^{-1} y_j}+\frac{1}{1-z x_i y_j^{-1}}+\frac{1}{1-z x_i^{-1} y_j^{-1}}.
\end{align*}
Expand $\det(X)$ to get (similar to proof of Theorem \ref{thm:SOodd})
\begin{align*}
\det(X) &= \det\left( \sum_{k\ge 0} z^k (x_i^{k }+x_i^{-k})
(y_j^{k }+y_j^{-k}) \right)\\
           &= \sum_{\pi \in S_m} sign(\pi) \prod_{i\le m}\left(\sum_{k\ge 0}
            z^k (x_i^{k }+x_i^{-k})(y_{\pi(i)}^{k }+y_{\pi(i)}^{-k})\right)\\
             &= \sum_{\pi \in S_m} sign(\pi) \sum_{a_1,a_2,\ldots,a_m\ge 0} \; \prod_{i\le m}z^{a_i} 
\left((x_i^{a_i }+x_i^{-(a_i )})(y_{\pi(i)}^{a_i }+y_{\pi(i)}^{-(a_i )})\right)\\
			&=  \sum_{a_1,a_2,\ldots,a_m\ge 0} \; z^{\sum_i a_i} \sum_{\pi \in S_m} sign(\pi) 
\prod_{i\le m}\left((x_i^{a_i }+x_i^{-(a_i )})(y_{\pi(i)}^{a_i }+y_{\pi(i)}^{-(a_i )})\right)\\
         &= \sum_{a \in \Z_+^m} z^{|a|} \det\left(y_{j}^{a_i }+y_{j}^{-(a_i )}\right) \prod_i (x_i^{a_i }+x_i^{-(a_i )}).
\end{align*}

Rewrite the last equality as follows by ordering the $m$-tuple $a\in \Z^m$.

\begin{align*}
\det(X)  = \sum_{a_1 \ge \ldots \ge a_m\ge 0}\; z^{|a|} \det\left(y_{j}^{a_i }+y_{j}^{-(a_i )}\right) 
          \det\left(x_{j}^{a_i }+x_{j}^{-(a_i )}\right).
\end{align*}

Clearly $a_1 > a_2> \ldots > a_m \ge 0$. Define $\lambda_i = a_i -m +i$, then $\lambda_1 \ge \ldots \ge \lambda_m\ge 0$. 
The last equality translates to the following in terms of $\lambda$:
\begin{align*}
\det(X) &= \sum_{\lambda_1 \ge \ldots \ge \lambda_m \ge 0}\; z^{|\lambda|+\binom{m}{2}} \det\left(A^\lambda\right) \det\left(B^\lambda\right),\\
A_{i,j}^\lambda &=y_{j}^{\lambda_i+m-i }+y_{j}^{-(\lambda_i+m-i )},\\
B_{i,j}^\lambda &= x_{j}^{\lambda_i+m-i }+x_{j}^{-(\lambda_i+m-i )}.
\end{align*}
Assume  $x_1,\ldots,x_m$ and $y_1,\ldots,y_m$ are such that there exist $g,h \in \mathit{SO}(2m)$ with 
eigenvalues equal to $x_1,x_1^{-1},\ldots,x_m,x_m^{-1}$ and  $y_1,y_1^{-1},\ldots,y_m,y_m^{-1}$, 
respectively. Dividing by a factor of the following form
\begin{align*}
z^{\binom{m}{2}} \det\left(y_{j}^{m-i }+y_{j}^{-(m-i )}\right) \det\left(x_{j}^{m-i }+x_{j}^{-(m-i )}\right),
\end{align*}
and using definition of $\chi_\lambda$ yields
\begin{align}\label{eq:SOEvenCauchyIntermidiate}
\sum_{\lambda}\; z^{|\lambda|} \chi_{\lambda}(g) \chi_{\lambda}(h) =
\frac{z^{-\binom{m}{2}}\det(X)}{ \det\left(y_{j}^{m-i }+y_{j}^{-(m-i )}\right) 
\det\left(x_{j}^{m-i }+x_{j}^{-(m-i )}\right)},
\end{align}
where $\lambda$ runs over all partitions of an arbitrary integer with at most $m$ parts.  To get to the final form note that
\begin{align*}
\det\left( x_j^{m-i}+x_j^{-(m-i)}\right)  &= \det\left( (x_j+x_j^{-1})^{m-i}\right)\\
																 &= \prod_{i<j} (x_i+x_i^{-1}-(x_j+x_j^{-1})) .
\end{align*}
Substituting these simplifications in (\ref{eq:SOEvenCauchyIntermidiate}) proves the theorem.
\end{proof}

\remark A similar $z \rightarrow 1$ analysis of Type C and Type D Cauchy identities is possible and gives rise to similar expressions as in Type B case. Carrying out the details is straightforward and omitted here.

\section{Discussion}
Although the identities introduced in the current paper provide a closed formula for an otherwise intractable infinite sum, they are still computationally expensive as they require evaluation of an $m\times m$ determinant for $m $ dimensional problems. Proving any alternative representations of these determinants which result in faster computations would be of significant practical importance. This is an interesting future direction to pursue. Another interesting direction to pursue would be the following. Find a conceptual proof of the Cauchy identity for compact groups similar to the one sketched for the unitary group. Note that an analogue of (\ref{CauchyCategorify}) holds for other compact groups. The main challenge is finding an appropriate basis for the coordinate ring of the group in order to compute the traces on both sides. Of course, any new applications in other parts of mathematics or applied problems are very much appreciated.

\section{Acknowledgments}
The author owes a great deal of gratitude to his doctoral advisor, Professor Persi Diaconis, for introducing him to the subject, suggesting the problem, and for his constant encouragement and support. The author is grateful to Professor Daniel Bump for illuminating discussions, and to Professor Arun Ram for sharing his notes on characters of classical groups and symmetric functions. We would like to thank Dr. Michael Wheeler for bringing to our attention their recently published work \citep{wheeler2016refined}.



\bibliographystyle{elsarticle-harv} 
\bibliography{CauchyIdentity}

\end{document}